\documentclass[a4j,11pt]{article}

\usepackage{natbib}
\usepackage[dvipdfmx,hiresbb]{graphicx}
\usepackage[dvipdfmx]{color}
\usepackage{amsmath}
\usepackage{amsthm}
\usepackage{ascmac}
\usepackage{amssymb}
\usepackage{color}
\usepackage{lastpage}
\usepackage{natbib}

\numberwithin{equation}{section}
\newtheorem{thm}{Theorem}
\newtheorem{cor}[thm]{Corollary}

\newtheorem{prop}[thm]{Proposition}
\newtheorem{rem}{Remark}

\def\ve{\varepsilon}
\def\var{\mbox{\rm var}}
\def\ne{{\rm e}}

\title{
A reversal phenomenon in estimation based on multiple samples from the Poisson--Dirichlet distribution
}
\author{Koji Tsukuda\footnote{Graduate School of Arts and Sciences, The University of Tokyo, 3-8-1 Komaba, Meguro-ku, Tokyo 153-8902, Japan. mail: ctsukuda@g.ecc.u-tokyo.ac.jp}; 
Shuhei Mano\footnote{Department of Mathematical Analysis and Statistical Inference, The Institute of Statistical Mathematics, 10-3 Midori-cho, Tachikawa, Tokyo, 190-8562, Japan. mail: smano@ism.ac.jp}}

\date{}

\begin{document}
\maketitle

\begin{abstract}
Consider two forms of sampling from a population: (i) drawing $s$ samples of $n$ elements with replacement and (ii) drawing a single sample of $ns$ elements.
In this paper, under the setting where the descending order population frequency follows the Poisson--Dirichlet distribution with parameter $\theta$, we report that the magnitude relation of the Fisher information, which sample partitions converted from samples (i) and (ii) possess, can change depending on the parameters, $n$, $s$, and $\theta$.
Roughly speaking, if $\theta$ is small relative to $n$ and $s$, the Fisher information of (i) is larger than that of (ii); on the contrary, if $\theta$ is large relative to $n$ and $s$, the Fisher information of (ii) is larger than that of (i).
The result represents one aspect of random distributions.
\\

\noindent MSC-2010. primary:62F10. secondary: 62D05, 92D10.  \\
key words and phrases: Bayesian nonparametrics; Fisher information; Microdata disclosure risk assessment; Population genetics; Random distribution.
\end{abstract}

\section{Introduction}
Let us consider the statistical modeling of a random sampling that produces data containing $ns$ records by drawing $s$ sets of $n$ elements with replacement, where $n$ and $s$ are positive integers.
Let $Z_{ij}$ be random variables corresponding $j(=1,\ldots,n)$-th element in $i(=1,\ldots,s)$-th sample.
For example, when we are interested in the population mean $\mu$ and the normal distribution ${\sf N}(\mu,1)$ with unit variance is used as a population model, the sample mean $\bar{Z}_{\cdot\cdot}$ of $ns$ records, which is a complete sufficient statistic, coincides with the sample mean of $s$ sample means $\bar{Z}_{1\cdot},\ldots,\bar{Z}_{n\cdot}$, which are calculated from each sample set of size $n$.
There exists a relation 
\[ ns \bar{Z}_{\cdot\cdot} = n \sum_{i=1}^s \bar{Z}_{i\cdot}, \]
and both samples have the same information for $\mu$.
Under the same setting, if we use the sample median $\tilde{Z}_{\cdot\cdot}$ of $ns$ records to estimate the population mean $\mu$ (which is identical to the population median in the case of ${\sf N}(\mu,1)$), which is not a sufficient statistic, have different information from one of which the sample mean $\bar{\tilde{Z}}_{\cdot\cdot}$ of $s$ sample medians $\tilde{Z}_{1\cdot},\ldots,\tilde{Z}_{n\cdot}$, which are calculated from each sample set of size $n$.
This implies that when to pool affects the efficiency of inference.
However, both of the asymptotic distributions of $\tilde{Z}_{\cdot\cdot}$ and $\bar{\tilde{Z}}_{\cdot\cdot}$ as $n\to\infty$ are ${\sf N}(\mu,\pi/(2ns))$, which indicates that they asymptotically have the same Fisher information, and so the efficiencies of inference asymptotically coincide.

In this paper, we report a more curious phenomenon that the magnitude relation of the Fisher information reverses depending on the values of parameters.
For the purpose of inferring the population diversity, the Poisson--Dirichlet distribution \citep{RefK} with parameter $\theta$ is assumed as a statistical model of the descending order population frequency.
If $\theta$ is known, the homozygosity ($1/(1+\theta)$), the probability of discovering a new species ($\theta/(n+\theta)$ where $n$ is the sample size), the expected number of singletons (population uniques) in a finite population ($n_0 \theta/(n_0 +\theta-1)$ where $n_0$ is the population size), and so on can be determined.
Moreover, the Poisson--Dirichlet distribution prior (or the Dirichlet process prior) plays an important role in the nonparametric Bayesian estimation when there exist possibly infinite labels. 
For example, see \cite{RefC} and its discussions; \cite{RefE,RefHT,RefSa}.
Therefore, estimating $\theta$ efficiently is a problem and we will show that sampling procedure influence the magnitude relation of the Fisher information, so our result can be applied to considering sampling procedures for estimating $\theta$ efficiently.

\begin{rem}
Professor Masaaki Sibuya pointed out that our result represents an aspects of differences between independently and identically distributed (iid) sequences and non-iid sequences (in this paper exchangeable sequences are considered).
\end{rem}

\section{Sampling from the Poisson--Dirichlet distribution}

\subsection{Sampling procedures}
Assume that the descending order frequency of the infinite population follows the Poisson--Dirichlet distribution with parameter $\theta$.
For details about the properties of the Poisson--Dirichlet distribution and related distributions, we refer the reader to \cite{RefF10}.
In this paper, the following two types of samples are considered:

\begin{itemize}
\item \textbf{Sample (i)}\\
Drawing $s$ samples of $n$ elements with replacement.
The numbers of distinct components in the respective sample partitions are denoted by $X_1,\ldots,X_s$.
\item \textbf{Sample (ii)}\\
Drawing a single sample of $ns$ elements.
The number of distinct components in the sample partition is denoted by $Y$.
\end{itemize}

If a sample of $m$ elements is drawn from a population that follows the Poisson--Dirichlet distribution with parameter $\theta$, then we have a sample frequency.
When the population diversity is in interest, we convert the sample frequency to a sample partition, and the sample partition follows the Ewens sampling formula with parameter $(m,\theta)$.
Then, the number of distinct components in the sample partition, which is the number of labels in the sample, follows the falling factorial distribution with parameter $(m,\theta)$ (this distribution is described in the next subsection).
Note that the number of distinct components in the sample is the complete sufficient statistic for $\theta$.
Hence, $X_1,\ldots,X_s$ independently and identically follow the falling factorial distribution with parameter $(n,\theta)$, and $Y$ follows the falling factorial distribution with parameter $(ns,\theta)$.

Even though the total numbers of elements are $ns$ in both sampling processes, Samples (i) and (ii) contain different information.
In this paper, we calculate the Fisher information acquired from the samples obtained using these two sampling processes and demonstrate that the magnitude relation of the Fisher information can be reversed based on the values of the parameters $n$, $s$ and $\theta$.

In the literature, there is another model called the multi-observer ESF \citep{RefBT}, where the 2-observer version is considered in \cite{RefECLW}.
In the model, a sample of size $n_0$ consists of $S$ subsamples whose sizes are $n_1,\ldots,n_S$ where $n_1+\cdots+n_S=n_0$.
The motivation of \cite{RefBT} is testing of the spatial homogeneity of tumors, which means the parameters $\theta$ are common in the left side and right side of a tumor.
In the tumor example, when a common underlying evolution process is considered, Sample (i) fits the situation because the population frequencies are considered to be different even in the same side (see the data in Table 2.1 of \cite{RefBT}, or its original methylation patterns  data in Fig 2 of \cite{RefSMWTS}).

\begin{rem}
As for the multi-observer ESF, the marginal distributions of subsamples are also the Ewens sampling formula.
Even if $s$ subsamples in the whole $S$ subsamples are considered with $n_1+\cdots+n_s << n_0$, $s$ subsamples are not regarded as independent.
However, they are exchangeable, and conditionally iid when the population frequency is fixed to one realization.
On the other hand, our Sample (i) corresponds to the case where there are different $s$ population frequencies.
\end{rem}

\begin{rem}\label{remH}
In the context of microdata disclosure risk assessments, the Ewens sampling formula has been used as a model for the frequency of frequencies, which is sometimes called the size index.
See, for example, \cite{RefHT} and \cite{RefSa}.
The problem considered here fits the situation where an investigation is conducted in $s$ areas and $n$ individuals are surveyed in each area with an assumption on the Poisson--Dirichlet hyperpopulation.
Sample (i) corresponds to the case where there are different $s$ population frequencies for each area, and Sample (ii) corresponds to the case where there is only one population frequency.
When collecting survey data, it is natural to consider that sample sizes are different area-by-area.
However, to clearly observe the reversal phenomenon of the Fisher information, we assume that the sample sizes of all areas are the same.
\end{rem}

\begin{rem}
Consider there exists possibly infinite number of labels, and we have sample frequency with finite number of labels.
In such situation, the Poisson-Dirichlet distribution can be used as a prior distribution, and in order to conduct empirical Bayes estimation its parameter should be estimated.
Suppose there are $s$ samples of size $n$.
If the samples are drawn from different populations which independently follow the Poisson--Dirichlet distribution then Sample (i) is suitable; on the other hand if the samples are drawn from one population which follows the Poisson--Dirichlet distribution then Sample (ii) is suitable.
\end{rem}

\subsection{The falling factorial distribution}

If the probability distribution of a positive-integer-valued random variable $X$ is given by
\[ {\rm pr}( X =x) = f(x,\theta) = \frac{\bar{s}(n,x) \theta^x}{(\theta)_n}, \quad (x=1,2,\ldots,n), \]
then the distribution of $X$ is sometimes called the falling factorial distribution with parameter $(n,\theta)$ \citep{RefW1}, where $(\theta)_n = \theta (\theta+1) \cdots (\theta+n-1)$ and $\bar{s}(n,x)$ is the coefficient of $\theta^x$ in $(\theta)_n$.
The falling factorial distribution is also called STR1F in \cite{RefS0}.
In a random partition following the Ewens sampling formula with parameter $(n,\theta)$, the number of distinct components follows the falling factorial distribution with parameter $(n,\theta)$ \citep{RefE}.
Since the number of distinct components is the complete sufficient statistic for $\theta$, considering the falling factorial distribution is enough in order to estimate $\theta$.
In this subsection, we describe the known properties of a random variable $X$ which follows the falling factorial distribution with parameter $(n,\theta)$.

The moment generating function of $X$ is given as
\[ E[\ne^{X t}] 
= \sum_{i=1}^n  \frac{\bar{s}(n,i) (\theta \ne^t)^i}{(\theta)_n}
= \frac{(\theta \ne^t)_n}{(\theta)_n} 
= \frac{\Gamma(\theta \ne^t +n) \Gamma(\theta)}{\Gamma(\theta+n)\Gamma(\theta \ne^t)}, \]
where $\Gamma(\cdot)$ is the gamma function.
Letting
\[ L_n(\theta)= \sum_{i=1}^n \frac{1}{\theta+i-1},
\quad
\ell_n(\theta) = \sum_{i=1}^n \frac{i-1}{(\theta+i-1)^2}, \]
the mean and variance of $X$ can be written as 
\[ E[X] = \theta L_n(\theta), \quad
 \var(X) = \theta \ell_n(\theta). \]

Let 
\[ l_\theta(x) = \log (f(x,\theta)) = \log\left( \frac{\bar{s}(n,x) \theta^x}{(\theta)_n} \right). \]
Then, the log-likelihood is $l_\theta(X)$ and its derivative $\dot{l}_{\theta}(X)$ with respect to $\theta$ is
\[ \dot{l}_\theta(X) = \frac{\partial}{\partial \theta} l_\theta(X) = \frac{X}{\theta} - \sum_{i=1}^n \frac{1}{\theta+i-1}. \]
Since $E[\dot{l}_\theta(X)]=0$, the Fisher information of $\theta$ is
\[ E \left[\left( \dot{l}_\theta (X) \right)^2 \right]  = \frac{1}{\theta^2} \var(X) = \frac{\ell_n(\theta)}{\theta} . \]

The maximum likelihood estimator $\hat\theta$ of $\theta$ is the root of
\[  X - \sum_{i=1}^n \frac{\hat\theta}{\hat\theta+i-1} = 0. \]
It is well known that if $\theta$ is fixed then $\hat\theta$ enjoys the asymptotic normality as $n\to \infty$:
\[ \sqrt{\log{n}} (\hat\theta - \theta) \Rightarrow {\sf N} \left( 0, \theta \right) \]
as $n\to\infty$, where $\Rightarrow$ denotes the convergence in distribution.
Moreover, it may be useful for some situation to consider the case where $\theta$ grows with $n$ \citep{RefF07,RefGT,RefT}. 
In particular, the asymptotic normality of $\hat\theta$ has been extended to the asymptotic case where both $n$ and $\theta$ tend to infinity \citep{RefT}:
\[ \sqrt{\frac{\ell_n(\theta)}{\theta}} (\hat\theta - \theta) \Rightarrow {\sf N} \left( 0, 1 \right) \]
as $n,\theta\to\infty$ with $n^2/\theta \to \infty$, or as $n\to\infty$ leaving $\theta>0$ fixed.
The asymptotic variance $\theta/\ell_n(\theta)$ of $\hat\theta$ generally does not become $\theta/\log{n}$.
Indeed if $n/\theta\to\infty$ then $\ell_n(\theta)\sim\log(n/\theta)$ and if $n/\theta\to0$ then $\ell_n(\theta)\sim n^2/(2\theta^2)$, where $n$ and $\theta$ grow simultaneously.
Therefore, it holds that if $n/\theta \to \infty$ then $\theta/\ell_n(\theta)\sim\theta/\log(n/\theta)$; further, if $n/\theta \to 0$ then $\theta/\ell_n(\theta)\sim2\theta^3/n^2$.

\section{The Fisher information of two samples}

In this section, the Fisher information of Samples (i) and (ii) is calculated, and the maximum likelihood estimators are presented.

\subsection{Sample (i)}
Since $X_1,\ldots,X_s$ independently and identically follow the falling factorial distribution with a parameter $(n,\theta)$, the likelihood is given by
\[ \prod_{k=1}^s f(X_k,\theta). \]
Thus, the log-likelihood $M_\theta$ is
\[M_\theta =  \sum_{k=1}^s l_\theta(X_k), \]
so its derivative with respect to $\theta$ is
\[\dot{M}_\theta 
= \sum_{k=1}^s \dot{l}_\theta(X_k)  
= \sum_{k=1}^s \left( \frac{X_k}{\theta} - \sum_{i=1}^n \frac{1}{\theta+i-1} \right) 
= \sum_{k=1}^s \frac{X_k}{\theta} - \sum_{i=1}^n \frac{s}{\theta+i-1}.
\]
It follows from $E[\dot{M}_\theta]=0$ and from the independence of $\{ \dot{l}_\theta(X_k) \}_{k=1}^s $ that the Fisher information $\mathcal{I}_1(\theta;n,s)$ with respect to $\theta$ is
\[ \mathcal{I}_1(\theta;n,s)
= E \left[\left( \dot{M}_\theta \right)^2 \right]  
= \var(\dot{M}_\theta)
=  \sum_{k=1}^s \var \left(  \dot{l}_\theta(X_k) \right)
=  \frac{s\ell_n(\theta) }{\theta}. \]

The maximum likelihood estimator $\hat\theta^{(1)}$ is the root of
\[\bar{X} - \sum_{i=1}^n \frac{\hat\theta^{(1)}}{\hat\theta^{(1)}+i-1}  = 0, \]
where $\bar{X}= s^{-1} T$ denotes the sample mean of $\{X_k \}_{k=1}^s $ where
\begin{equation}
T= \sum_{i=1}^s X_i . \label{defT}
\end{equation}

\begin{rem}\label{remclt}
The moment generating function of $T$ is
\begin{equation}
E \left[ \ne^{T t} \right] 
= \left(  \frac{ (\theta \ne^{t})_n }{(\theta)_n} \right)^s 
= \prod_{j=1}^n \left( \frac{\theta }{\theta+j-1}\ne^t + \frac{j-1 }{\theta+j-1} \right)^s. \label{mgfT}
\end{equation}
Therefore, when $s\neq 1$, $T$ does not follow the falling factorial distribution.
Moreover, when $n\to\infty$, if and only if $sn^2/\theta \to \infty$, 
\[ \frac{T - s \theta L_n(\theta)}{ \sqrt{s\theta \ell_n(\theta)} }  \Rightarrow {\sf N}(0,1) \]
(see Proposition \ref{PAN} in Appendix).
\end{rem}

\subsection{Sample (ii)}
Since $Y$ follows the falling factorial distribution with parameter $(ns,\theta)$, the Fisher information $\mathcal{I}_2(\theta;n,s)$ with respect to $\theta$ is given as
\[\mathcal{I}_2(\theta;n,s) = \frac{1}{\theta} \sum_{i=1}^{ns} \frac{i-1}{(\theta+i-1)^2} = \frac{\ell_{ns}(\theta)}{\theta} . \]

The maximum likelihood estimator $\hat\theta^{(2)}$ is the root of
\[ Y - \sum_{i=1}^{ns} \frac{\hat\theta^{(2)}}{\hat\theta^{(2)}+i-1} = 0. \]

\section{Comparing the two samples}
As is stated before, the Fisher information of Samples (i) and (ii) --$\mathcal{I}_1(\theta;n,s)$ and $\mathcal{I}_2(\theta;n,s)$, respectively-- are different even though the total numbers of elements in the two Samples are the same.
Considering simple asymptotic settings, the asymptotic magnitude relation between $\mathcal{I}_1(\theta;n,s)$ and $\mathcal{I}_2(\theta;n,s)$ can be observed as follows:
If $n\to\infty$ leaving $s(\geq2)$ and $\theta$ fixed, then
\[ \mathcal{I}_1(\theta;n,s) \sim \frac{s \log{n}}{\theta}  ,\quad \mathcal{I}_2(\theta;n,s) \sim \frac{\log{n}}{\theta} , \]
so $\mathcal{I}_1(\theta;n,s) > \mathcal{I}_2(\theta;n,s)$ asymptotically;
if $s\to\infty$ leaving $n(\geq2)$ and $\theta$ fixed, then
\[ \mathcal{I}_1(\theta;n,s) = \frac{s \ell_n(\theta)}{\theta}  ,\quad \mathcal{I}_2(\theta;n,s) \sim \frac{\log{s}}{\theta} , \]
so $\mathcal{I}_1(\theta;n,s) > \mathcal{I}_2(\theta;n,s)$ asymptotically.
On the contrary, if $n=1$ and $s\geq2$, $\mathcal{I}_1(\theta;1,s) = 0 <\mathcal{I}_2(\theta;1,s)$,  so $\mathcal{I}_1(\theta;n,s) < \mathcal{I}_2(\theta;n,s)$.
These observations indicate that the magnitude relation can be reversed based on the values of the parameters $n,s$ and $\theta$.
To demonstrate this phenomenon more precisely, non-asymptotic sufficient conditions that guarantee the magnitude relation between $\mathcal{I}_1(\theta;n,s)$ and $\mathcal{I}_2(\theta;n,s)$ are provided in Subsection \ref{NAR}.
Moreover, the asymptotic conditions that guarantee the asymptotic magnitude relation where a pair of parameters either $(n,\theta)$, $(s,\theta)$, or $(n,s)$, grows simultaneously leaving the rest parameter fixed are provided in Subsection \ref{AR}.

\subsection{Non-asymptotic results}\label{NAR}
In this subsection, let us provide sufficient conditions under which $s\ell_n(\theta) > \ell_{ns}(\theta)$ and a sufficient condition under which $s\ell_n(\theta) < \ell_{ns}(\theta)$.
The results indicate that, roughly speaking, the former inequality holds when $\theta$ is relatively small in comparison to $n$ and $s$, and the latter inequality holds when $\theta$ is relatively large in comparison to that relative to $n$ and $s$.

Firstly, we provide a sufficient condition under which $s\ell_n(\theta) > \ell_{ns}(\theta)$.

\begin{thm}\label{th0}
(1) When $\theta\geq2$, if integers $s$ and $n$ satisfy $s\geq2$ and
\begin{equation}
n > 1+ \frac{\lfloor \theta \rfloor  -1 + \frac{1}{s}}{\ell_{\lfloor \theta \rfloor}(\theta)} , \label{cop2}
\end{equation}
then
$s \ell_n(\theta) > \ell_{ns}(\theta)$.\\
(2) When $\theta<2$, if integers $s$ and $n$ satisfy $s\geq2$ and
\begin{equation}
n > 1+(\theta+1)^2\left(\frac{1}{s} + 1\right), \label{cop3}
\end{equation}
then
$s \ell_n(\theta) > \ell_{ns}(\theta)$.
\end{thm}

\begin{proof}
First, note that
$x \mapsto x(\theta+x)^{-2}$
is decreasing as $x$ increases in the range $x>\theta$ because its derivative is
$x \mapsto (\theta-x)(\theta+x)^{-3}$.
Moreover, either the hypothesis of (1) or (2) implies $n>\theta$, since
\begin{eqnarray*}
1+ \frac{\lfloor \theta \rfloor  -1 + \frac{1}{s}}{\ell_{\lfloor \theta \rfloor}(\theta)} -\theta 
&>& 1+ \frac{\lfloor \theta \rfloor  -1 + \frac{1}{s}}{1/2} -\theta
= 2\lfloor \theta \rfloor  -1 + \frac{2}{s} -\theta \\
&\geq& \lfloor \theta \rfloor  -2 + \frac{2}{s},
\end{eqnarray*}
where
\[ \ell_{\lfloor \theta \rfloor}(\theta) 
< \frac{\lfloor \theta \rfloor(\lfloor \theta \rfloor-1)}{2\theta^2} 
< \frac{\theta ( \theta -1)}{2\theta^2} 
< \frac{1}{2} \]
are used in the first inequality and
\[ 1+(\theta+1)^2\left(\frac{1}{s} + 1\right) > \theta. \]
Therefore,
\begin{eqnarray}
&& s \ell_n(\theta) - \ell_{ns}(\theta) \nonumber \\
&=& (s-1) \sum_{i=2}^{n} \frac{i-1}{(\theta+i-1)^2} - \sum_{i=n+1}^{ns-s+1} \frac{i-1}{(\theta+i-1)^2} - \sum_{i=ns-s+2}^{ns} \frac{i-1}{(\theta+i-1)^2} \nonumber \\
&>&  (s-1) \sum_{i=2}^n \frac{i-1}{(\theta+i-1)^2} - (ns-n-s+1) \frac{n}{(\theta+n)^2} \\&& \quad - (s-1) \frac{ns -s +1}{(\theta+ns -s +1)^2} \nonumber \\
&=& (s-1) \left[ \sum_{i=2}^n \left\{ \frac{i-1}{(\theta+i-1)^2} -  \frac{n}{(\theta+n)^2} \right\} - \frac{ns -s +1}{(\theta+ns -s +1)^2} \right]. \label{RH1}
\end{eqnarray}
Hereafter, the assertions (1) and (2) are considered separately.

(1)
When $\theta\geq 2$, the right-hand side of \eqref{RH1} is larger than
\[
(s-1) \left[ \sum_{i=2}^{\lfloor \theta \rfloor} \left\{ \frac{i-1}{(\theta+i-1)^2} -  \frac{n}{(\theta+n)^2} \right\} - \frac{ns -s +1}{(\theta+ns -s +1)^2} \right].
\]
Further, if the hypothesis is true then this display is positive, because the terms enclosed in the brackets are
\begin{eqnarray*}
&&\sum_{i=2}^{\lfloor \theta \rfloor} \left\{ \frac{i-1}{(\theta+i-1)^2} -  \frac{n}{(\theta+n)^2} \right\} - \frac{ns -s +1}{(\theta+ns -s +1)^2} \\
&=& \ell_{\lfloor \theta \rfloor} (\theta) - (\lfloor\theta\rfloor-1) \frac{n}{(\theta+n)^2}  - \frac{ns -s+ 1}{(\theta+ns - s +1)^2} \\
&>& \ell_{\lfloor \theta \rfloor} (\theta) - (\lfloor\theta\rfloor-1) \frac{1}{n-1}  - \frac{1}{s(n - 1) } \\
&=& \ell_{\lfloor \theta \rfloor} (\theta) - \left( \lfloor\theta \rfloor -1 + \frac{1}{s} \right) \frac{1}{n-1}.
\end{eqnarray*}

(2)
When $\theta< 2$, the right-hand side of \eqref{RH1} is larger than
\[
(s-1) \left\{ \frac{1}{(\theta+1)^2} -  \frac{n}{(\theta+n)^2} - \frac{ns -s +1}{(\theta+ns -s +1)^2} \right\}.
\]
Further, if the hypothesis is true then this display is positive, because the terms enclosed in the brackets are
\begin{eqnarray*}
&& \frac{1}{(\theta+1)^2} -  \frac{n}{(\theta+n)^2} - \frac{ns -s +1}{(\theta+ns -s +1)^2} \\
&>& \frac{1}{(\theta+1)^2} -  \frac{1}{n-1} - \frac{1}{s(n -1)} \\
&=& \frac{1}{(\theta+1)^2} -  \left( 1 + \frac{1}{s} \right) \frac{1}{n-1}.
\end{eqnarray*}

This completes the proof. 
\end{proof}

\begin{rem}
To evaluate \eqref{cop2}, some values of
\begin{equation} 
1+ \frac{\lfloor \theta \rfloor  -1 + \frac{1}{s}}{\ell_{\lfloor \theta \rfloor}(\theta)}  \label{conr}
\end{equation}
are shown in Table~\ref{conv}.
The condition given in Theorem \ref{th0} is only a sufficient condition; when $(n,s,\theta)=(7,2,3)$, $s\ell_n(\theta) - \ell_{ns}(\theta) = 0.042>0$.
\end{rem}

\begin{table}[t]
\centering
\caption{Values of \eqref{conr}}
  \begin{tabular}{cccccccc}
 \hline
& $\theta=$ & $2$ & $3$ & $4$ & $5$ & $10$ & $100$ \\ \hline
$s=2$ & & 13.50 & 17.54 & 22.32 & 27.30 & 52.83 & 518.53 \\ 
$s=3$ & & 12.00 & 16.37 & 21.26 & 26.29 & 51.90 & 517.66 \\ 
$s=10$ & & 9.90 & 14.74 & 19.77 & 24.87 & 50.61 & 516.44 \\ 
\hline
  \end{tabular}
\label{conv}
\end{table}

If $\theta\leq1$ then \eqref{cop3} is satisfied for $n \geq 8$ and $s\geq2$.
Therefore, we have the following corollary.

\begin{cor}
If $\theta\leq1$, then $\mathcal{I}_1(\theta;n,s) > \mathcal{I}_2(\theta;n,s)$ for any integers $n\geq8$ and $s\geq2$.
\end{cor}

Next, we provide another sufficient condition under which $s\ell_n(\theta) > \ell_{ns}(\theta)$.

\begin{thm}\label{thn}
If
\begin{equation}
1 \leq \theta \leq \left( \frac{s}{\log(1+ns)} \right)^{1/2} - 1, \label{assn1}
\end{equation}
then $s \ell_n(\theta) > \ell_{ns}(\theta)$ for any integers $n\geq2$ and $s\geq2$.
\end{thm}

\begin{proof}
From
\[ \ell_n(\theta) = L_n(\theta) - \theta \sum_{i=1}^n \frac{1}{(\theta+i-1)^2}, \]
\[ \sum_{i=1}^n \frac{1}{(\theta+i-1)^2} \geq \frac{n}{\theta(n+\theta)}, \]
and
\[ L_n(\theta) < \log\left(1+\frac{n}{\theta} \right) + \frac{n}{\theta(\theta+n)} \]
(see the proof of Proposition 1 in \cite{RefT} for derivations of these inequalities), it follows that
\[ \ell_n(\theta) < \log\left(1+\frac{n}{\theta} \right) + \left(\frac{1}{\theta} - 1\right)\frac{n}{\theta+n}. \]
When $\theta\geq1$, this display implies that
\begin{eqnarray}
\ell_{ns} (\theta) 
&<& \log\left(1+\frac{ns}{\theta} \right) \label{pa20} \\
&\leq& \log\left(1+ ns \right). \label{pa2}
\end{eqnarray}
On the other hand, when $n\geq2$, it holds that
\begin{equation}
s \ell_{n} (\theta) \geq \frac{s}{(\theta+1)^2} . \label{pa1}
\end{equation}
Hence, \eqref{pa2} and \eqref{pa1} imply that
\[ s \ell_n(\theta) - \ell_{ns}(\theta) > \frac{s}{(\theta+1)^2} - \log\left(1+ ns \right) \geq 0, \]
where the last inequality follows from the assumption \eqref{assn1}.
This completes the proof.
\end{proof}

\begin{rem}\label{remn2}
By using \eqref{pa20} instead of \eqref{pa2}, we can loose \eqref{assn1} to
\[
1 \leq \theta \leq \left( \frac{s}{\log\left(1+{ns}/{\theta}\right)} \right)^{1/2} - 1. 
\]
\end{rem}

\begin{rem}
By assuming more direct condition, we may have sharper sufficient conditions.
For example, if
\begin{equation}\label{dc1}
\ell_n(\theta) >  \frac{\log{s}}{s+1} - \frac{n\theta}{(n+1+\theta)(ns+1+\theta)}
\end{equation}
then $s\ell_n(\theta) - \ell_{ns}(\theta)>0$.
That is because 
\begin{eqnarray*}
s\ell_n(\theta) - \ell_{ns}(\theta) 
&>& (s-1) \ell_n(\theta)+ \log\left(\frac{n+\theta}{ns+\theta} \right) + \frac{(s-1)n\theta}{(n+1+\theta)(ns+1+\theta)} \\
&>& (s-1)\left\{ \ell_n(\theta)+ \frac{n\theta}{(n+1+\theta)(ns+1+\theta)} - \frac{\log{s}}{s+1} \right\},
\end{eqnarray*}
where we have used
\begin{eqnarray*}
\sum_{i=a}^b \frac{i}{(i+\theta)^2} 
&=& \sum_{i=a}^b \left\{ \frac{1}{i+\theta} - \frac{\theta}{(i+\theta)^2} \right\} \\
&<& \log\left( \frac{b+\theta}{a-1+\theta} \right) - \frac{\theta(b-a+1)}{(a+\theta)(b+1+\theta)} 
\end{eqnarray*}
for positive integers $a,b \ (a\leq b)$.
However, interpreting \eqref{dc1} in the current form may be difficult.
\end{rem}

Finally, we provide a sufficient condition under which $s\ell_n(\theta) < \ell_{ns}(\theta)$.

\begin{thm} \label{th1}
If
\begin{equation}
\theta > \sqrt{(n-1)(ns+n-1)},  \label{cop}
\end{equation}
then $s \ell_n(\theta) < \ell_{ns}(\theta)$ for any integers $n\geq1$ and $s\geq2$.
\end{thm}

\begin{proof}
Define the function $g:s\mapsto g(s)$ as
\[ g(s) =  s\ell_n(\theta) - \ell_{ns}(\theta). \]
First, when $n=1$, it holds that $g(s) =  - \ell_{s}(\theta) < 0$ for $s=2,3,\ldots$.
Next, consider $n\geq2$.
Since $g(1)=0$, it is sufficient to show that $g(s+1) < g(s)$ when \eqref{cop} is satisfied.
It holds that
\begin{eqnarray*}
&& g(s+1) - g(s) \\
&=& \ell_n(\theta) - \ell_{n(s+1)}(\theta) + \ell_{ns}(\theta) \\
&=& \sum_{i=1}^n \frac{i-1}{(\theta+i-1)^2} - \sum_{i=ns+1}^{ns+n} \frac{i-1}{(\theta+i-1)^2} \\
&=& \sum_{i=1}^n\left\{ \frac{i-1}{(\theta+i-1)^2} - \frac{ns + i-1}{(\theta+ns+i-1)^2} \right\} \\
&=& \sum_{i=1}^n\left\{ \frac{(i-1)(\theta+ns+i-1)^2 - (ns + i-1)(\theta+i-1)^2}{(\theta+i-1)^2(\theta+ns+i-1)^2} \right\}, 
\end{eqnarray*}
where the numerator of the fraction contained in the brackets on the right-hand side is negative since
\[
 ns \left\{ -\theta^2 + (i-1)(ns + i - 1)  \right\} 
\leq  ns \left\{ -\theta^2 + (n-1)(ns + n - 1)  \right\} 
< 0.
\]
Therefore it holds that $g(s+1) < g(s)$, which implies $g(s) < g(1) = 0$ for $s = 2,3,\ldots$.
This completes the proof.
\end{proof}

\subsection{Asymptotic results}\label{AR}

In this subsection, we consider situations where two of $n,s,$ and $\theta$ tend to infinity simultaneously leaving the rest one fixed.
The results presented in this subsection are summarized in Table~\ref{asytab}, where $\mathcal{I}_1(\theta;n,s)$ and $\mathcal{I}_2(\theta;n,s)$ are denoted by $\mathcal{I}_1$ and $\mathcal{I}_2$, respectively.

\begin{table}[t]
\centering
\caption{Asymptotic values of $s\ell_n(\theta)$ and $\ell_{ns}(\theta)$ and the asymptotic magnitude relations of the corresponding Fisher information}
  \begin{tabular}{lccc}
 \hline
    asymptotic setting & $s\ell_n(\theta)$ & $\ell_{ns}(\theta)$ & magnitude relation \\ \hline
(I)		$n/\theta\to\infty$, $s$:fixed & $s\log(n/\theta)$ & $\log(n/\theta)$ & $\mathcal{I}_1 > \mathcal{I}_2$ \\
(II)		$n/\theta\to0$, $s$:fixed  & $n^2s/(2\theta^2)$ & $n^2s^2/(2\theta^2)$ & $\mathcal{I}_1 < \mathcal{I}_2$ \\ \hline
(III)		$s/\theta\to\infty$, $n$:fixed  & $n^2s/(2\theta^2)$ & $\log(s/\theta)$ & case-by-case \\
(IV)		$s/\theta\to c\neq0$, $n$:fixed  & $n^2s/(2\theta^2)$ & $\ell_{ns}(\theta)$ & $\mathcal{I}_1 < \mathcal{I}_2$ \\
(V)		$s/\theta\to0$, $n$:fixed  & $n^2s/(2\theta^2)$ & $n^2s^2/(2\theta^2)$ &  $\mathcal{I}_1 < \mathcal{I}_2$ \\ \hline
(VI)		$n,s\to\infty$, $\theta$:fixed  & $s\log{n}$ & $\log(ns)$ &  $\mathcal{I}_1 > \mathcal{I}_2$ \\ 
\hline
  \end{tabular}
\label{asytab}
\end{table}

Six different cases with various asymptotic settings are considered as follows:

(I) $n/\theta \to \infty$, leaving $s\geq2$ fixed\\
Since
\[s \ell_n(\theta) \sim s \log\left( \frac{n}{\theta} \right), \quad
\ell_{ns}(\theta) \sim  \log\left( \frac{n}{\theta} \right), \]
it follows that $s \ell_n(\theta)>\ell_{ns}(\theta)$ asymptotically.
This corresponds to Theorem \ref{th0} since $\ell_{\lfloor \theta \rfloor}(\theta) \asymp 1$ as $\theta \to \infty$.\\

(II) $n/\theta\to0$, leaving $s\geq2$ fixed\\
Since
\[s \ell_n(\theta) \sim s \frac{n^2}{2\theta^2}, \quad
\ell_{ns}(\theta) \sim  \frac{n^2 s^2}{2\theta^2}, \]
it follows that $s \ell_n(\theta) < \ell_{ns}(\theta)$ asymptotically.
This corresponds to Theorem \ref{th1}.\\

(III) $s/\theta \to \infty$, leaving $n\geq2$ fixed\\
Since
\[s \ell_n(\theta) \sim s \frac{n^2}{2\theta^2}, \quad
\ell_{ns}(\theta) \sim  \log\left( \frac{s}{\theta} \right), \]
there are three possible cases that differ in terms of the asymptotic magnitude relation between $s$ and $\theta^2\log{(s/\theta)}$:
(a) If $s/(\theta^2\log{(s/\theta)})\to\infty$ then $s \ell_n(\theta) > \ell_{ns}(\theta)$ asymptotically, which corresponds to Theorem \ref{thn} (see Remark \ref{remn2}).
(b) If $s/(\theta^2\log{(s/\theta)})\to0$ then $s \ell_n(\theta) < \ell_{ns}(\theta)$ asymptotically, which partly corresponds to Theorem \ref{th1}.
(c) If $s \asymp \theta^2\log{(s/\theta)}$ then 
\[s \ell_n(\theta) \sim \frac{n^2 K}{2} \log\theta, \quad
\ell_{ns}(\theta) \sim  \log\theta, \]
where $K = \lim s/ (\theta^2\log{(s/\theta)} ) $; therefore, in this case, the asymptotic magnitude relation between $s \ell_n(\theta)$ and $\ell_{ns}(\theta)$ is determined by the magnitude relation between $K$ and $2/n^2$.\\

(IV) $s \asymp \theta$, leaving $n\geq2$ fixed\\
Since
\[s \ell_n(\theta) \sim s \frac{n^2}{2\theta^2}\to0\]
and
\[\liminf_{s,\theta} \ell_{ns}(\theta) > \liminf_{s,\theta} \left\{ \frac{ns(ns-1)}{2\theta^2} \frac{1}{(1+ns/\theta)^2} \right\} > 0, \]
which follows from 
\[ \ell_{ns}(\theta) > \frac{1}{(\theta+ns)^2} \sum_{i=1}^{ns} (j-1) = \frac{ns(ns-1)}{2\theta^2} \frac{1}{(1+ns/\theta)^2}, \]
it holds that $s \ell_n(\theta) < \ell_{ns}(\theta)$ asymptotically.
This corresponds to Theorem \ref{th1}.\\

(V) $s / \theta\to 0$, leaving $n\geq2$ fixed\\
Since
\[s \ell_n(\theta) \sim s \frac{n^2}{2\theta^2}, \quad
\ell_{ns}(\theta) \sim  \frac{n^2 s^2}{2\theta^2}, \]
it follows that $s \ell_n(\theta) < \ell_{ns}(\theta)$ asymptotically.
This corresponds to Theorem \ref{th1}.\\

(VI) $n,s \to \infty$, leaving $\theta$ fixed\\
Since
\[s \ell_n(\theta) \sim s \log{n} , \quad
\ell_{ns}(\theta) \sim  \log\left(ns \right), \]
it follows that $s \ell_n(\theta)>\ell_{ns}(\theta)$ asymptotically.
This corresponds to Theorem \ref{th0}.\\

\section{Concluding remarks}
Under the assumption of the Poisson--Dirichlet population, we have presented that two ways (i) and (ii) of sampling procedures have different information, that is to say, they lead totally different results.
The reason of this phenomenon is that the Ewens sampling formula represents the law of samples from the Poisson--Dirichlet population, a typical random discrete distribution.
By virtue of the de Finetti theorem, for an exchangeable sequence there exists a directing measure such that the sequence is conditionally iid.
Our result indicates that when data analyses using random distributions are conducted, it is crucial to decide whether the data in interest is sample from an unobservable identical directing measure or not.

\appendix
\section{Appendix: The asymptotic normality of $T$}
In this Appendix, the following proposition, which was mentioned in Remark \ref{remclt}, is proven.

\begin{prop}\label{PAN}
If
\begin{equation}
 \frac{sn^2}{\theta} \to \infty, \label{NSC}
\end{equation}
then
\begin{equation}
\frac{T - s \theta L_n(\theta)}{\sigma } \Rightarrow {\sf N}(0,1) \label{Tan}
\end{equation}
where $T$ is as defined in \eqref{defT} and $\sigma = \sigma(n,s,\theta) = (s \theta \ell_n(\theta) )^{1/2}$.
Moreover, when $n\to\infty$, \eqref{NSC} is also necessary for \eqref{Tan}.
\end{prop}

\begin{proof}
For $i=1,2,\ldots,ns$, let 
\[ p_i = \frac{\theta}{\theta+j}, \quad (i-1 \equiv j \pmod{n}). \]
Consider a triangular sequence $\{ \xi_i \}_{i=1}^{ns}$ of independent Bernoulli variables, where $E[\xi_i] = p_i$ for all $i=1,\ldots,ns$. 
Then, from \eqref{mgfT}, the distribution of $T$ is the same as the distribution of
\[ \sum_{i=1}^{ns} \xi_i. \]

First, we prove that \eqref{NSC} implies
\begin{equation}
\frac{ \sum_{i=1}^{ns} \xi_i - s \theta L_n(\theta)}{\sigma } \Rightarrow {\sf N}(0,1). \label{xic}
\end{equation}
From the central limit theorem for bounded random variables, it is sufficient to show that $\sigma\to\infty$.
When $n/\theta \to \infty$ or $n/\theta \to c>0$, it follows that $\sigma^2=s\theta \ell_n(\theta)\to\infty$.
When $n/\theta \to 0$, it follows from \eqref{NSC} that
\[ \sigma^2 = s\theta \ell_n(\theta) \sim \frac{sn^2}{2\theta} \to \infty. \]

Next, we prove that when $n\to\infty$ \eqref{xic} implies \eqref{NSC}.
Assume that \eqref{NSC} does not hold.
Consider $n,s,\theta$ such that $ns <\theta$.
Since $\sigma^2 \sim s n^2/(2\theta)$ and
\[ p_i(1-p_i) \leq \theta \frac{ns-1}{(\theta+ns-1)^2} \sim \frac{ns}{\theta}\]
for all $i$ which follows from the fact that $x(\theta+x)^{-2}$ is increasing for $x<\theta$, it holds that 
\[ \frac{1}{\sigma^{2}} \max_{1 \leq i\leq ns} \{ p_i (1-p_i) \} \to 0 . \]
Hence, in order for \eqref{xic} to hold, the Lindeberg condition
\begin{equation}
 \lim_{n,s,\theta} \frac{1}{\sigma^2}\sum_{i=1}^{ns} E\left[ |\xi_i - p_i|^2 1\{ |\xi_i - p_i|> \ve \sigma \} \right] =0 \label{LC}
\end{equation}
for any $\ve>0$ is necessary where $1\{ \cdot \}$ is the indicator function, so we see that \eqref{LC} does not hold.
Since
\[\sup_{n,s,\theta} \left( s \ell_n(\theta) \right) <\infty, \]
we can take
\[ \ve = \inf_{n,s,\theta} \left( \frac{\theta}{\sigma( \theta+n)} \right) >0 .\]
Then, we have
\[ \frac{p_i}{\sigma} > \frac{\theta}{\sigma(\theta+n)}  > \ve \]
for all $i=1,\ldots,ns$, which yields that
\begin{eqnarray*} 
&& E\left[ |\xi_i - p_i|^2 1\{ |\xi_i - p_i|> \ve \sigma \} \right] \\
&=& (1-p_i)^2 p_i 1\{ 1 - p_i> \ve \sigma \} + (p_i)^2 (1-p_i) 1\{  p_i> \ve \sigma \} \\
&\geq& (p_i)^2 (1-p_i).
\end{eqnarray*}
Therefore, we have
\begin{eqnarray*} 
&& \frac{1}{\sigma^2}\sum_{i=1}^{ns} E\left[ |\xi_i - p_i|^2 1\{ |\xi_i - p_i|> \ve \sigma \} \right] \\
&\geq& \frac{1}{\sigma^2} \sum_{i=1}^{ns}  (p_i)^2 (1-p_i) \\
&=& \frac{s}{\sigma^2} \sum_{j=0}^{n-1}  \left( \frac{\theta}{\theta+j} \right)^2 \left( \frac{j}{\theta+j} \right) \\
&>& \frac{s\theta^2}{\sigma^2} \sum_{j=0}^{n-1} \frac{j}{(\theta+n)^3} \\
&=& \frac{s \theta^2 n (n-1)}{2 (\theta+n)^3 \sigma^2 }.
\end{eqnarray*}
When \eqref{NSC} does not hold, it holds that $\sigma^2 \sim sn^2/(2\theta)$, which guarantees that the right-hand side of the above display can not converge to 0 because
\[ \frac{s \theta^2 n (n-1)}{2 (\theta+n)^3 \sigma^2 } \sim  \frac{ \theta^3 }{ (\theta+n)^3  } \not\to 0. \]

This completes the proof.
\end{proof}

Letting $s=1$ in Proposition \ref{PAN}, we have the following corollary, which is essentially equivalent to Theorem 2 of \cite{RefT} (the necessary part is given from Theorem 3 of \cite{RefT}).

\begin{cor}\label{CAN}
$ n^2/\theta \to \infty$ is equivalent to
\[ \frac{X - \theta L_n(\theta)}{\sqrt{\theta \ell_n(\theta)} } \Rightarrow {\sf N}(0,1), \]
where $X$ follows the falling factorial distribution with parameter $(n,\theta)$.
\end{cor}

\begin{rem}
When $n\to\infty$, Proposition \ref{PAN} yields that \eqref{Tan} holds only if $n^2s/\theta \to \infty$.
On the contrary, when $n\to\infty$, Corollary \ref{CAN} yields that 
\begin{equation}
\frac{ Y - \theta L_{ns}(\theta) }{\sqrt{\theta \ell_{ns} (\theta)}} \Rightarrow {\sf N}(0,1) \label{Yan}
\end{equation}
if (and only if) $n^2s^2/\theta \to \infty$, where $Y$ follows the falling factorial distribution with parameter $(ns,\theta)$.
Hence, even when \eqref{Tan} does not hold, \eqref{Yan} can hold.
\end{rem}

\section*{Acknowledgments}
The authors would like to express their gratitude to Professor Nobuaki Hoshino for explaining his motivation in disclosure risk assessments concerning Remark \ref{remH} and Professor Masaaki Sibuya for providing us valuable comments.
This work was supported by Grant-in-Aid for Scientific Research (B) 16H02791, from Japan Society for the Promotion of Science.

\end{document}